\documentclass[12pt,amstex,amssymb]{amsart}
\usepackage{graphicx}
\usepackage{tikz}
\usepackage{framed}
\usepackage{amsmath}
\usepackage{amssymb}

\def\cB{\mathcal{B}}
\def\cH{\mathcal{H}}

\def\cF{\mathcal{F}}

\def\cO{\mathcal{O}}
\def\cC{\mathcal{C}}
\def\cD{\mathcal{D}}
\def\cN{\mathcal{N}}
\def\cG{\mathcal{G}}
\def\cGL{\mathcal{G}\mathcal{L}}
\def\cS{\mathcal{S}}
\def\cT{\mathcal{T}}
\def\R{\mathbb R}
\def\T{\mathbb T}
\def\Z{\mathbb Z}
\def\C{\mathbb C}
\def\N{\mathbb N}

\def\b0{\mbox{\boldmath $0$}}

\newtheorem{prop}{Proposition}[section]
\newtheorem{thm}[prop]{Theorem}
\newtheorem{rem}[prop]{Remark}
\newtheorem{lemm}[prop]{Lemma}
\newtheorem{cor}[prop]{Corollary}

\theoremstyle{definition}
\newtheorem{definition}[prop]{Definition}

\title{Decomposing the wavelet representation for shifts by wallpaper groups}

\author{Lawrence W. Baggett}
\address{Department of Mathematics, University of Colorado}
\email{baggett@colorado.edu}

\author{Kathy D. Merrill}
\address{Department of Mathematics, Colorado College}
\email{kmerrill@coloradocollege.edu}

\author{Judith A. Packer}
\address{Department of Mathematics, University of Colorado}
\email{packer@Colorado.edu}

\author{Keith F. Taylor}
\address{Department of Mathematics \& Statistics, Dalhousie University}
\email{keith.taylor@dal.ca}

\begin{document}

\maketitle

\begin{abstract}
The wavelet group and wavelet representation associated with shifts coming from a two dimensional crystal symmetry group
$\Gamma$
and dilations by powers of 3, are defined and studied. The main result is an explicit decomposition of this 
$3\Gamma$-wavelet representation into irreducible representations
of the wavelet group.
\end{abstract}

\section{Introduction}
In \cite{mt} and \cite{gcm}, the classical concept of wavelets on $\R^n$ was
modified by replacing shifts by points from a lattice in $\R^n$
with ``shifts'' by the isometries from a crystal symmetry group, $\Gamma$.
In this generalization, the dilation matrix $A$
must be compatible with the action of $\Gamma$.
For $\gamma\in\Gamma$, let $R(\gamma)$ denote
the unitary shift operator on the Hilbert space
$L^2(\R^n)$ defined by
$R(\gamma)f(y)=f(\gamma^{-1}\cdot y)$,
for all $y\in\R^n$ and $f\in L^2(\R^n)$. Here
$\gamma^{-1}\cdot y$ denotes the result of
shifting $y$ by the isometry $\gamma^{-1}$. When
$\Gamma$ is the smallest crystal symmetry group 
consisting of translations by elements of $\Z^n$, then
$R$ gives the classical shifts.
The dilation unitary, $D_A$, is given by
$D_Af(y)=|\det(A)|^{1/2}f(Ay)$, for all $y
\in\R^n$ and $f\in L^2(\R^n)$.
 The purpose of the present paper is to initiate a study of the operator algebraic and
group representational structures 
of
$\cG(A,\Gamma)$, the smallest group of 
unitary operators containing both $D_A$ and
$R(\Gamma)$, in the case where $\Gamma$ is a wallpaper group;
that is, a symmetry group of a two dimensional crystal. We
take the point of view, that $\cG(A,\Gamma)$ is the
image of a unitary representation of a particular group
we define in section 2.

This paper partly generalizes \cite{lpt},
where such a study was carried out when the
shifts were from the integer lattice, $\Z^n$. In that case,
the corresponding representation was shown to be unitarily
equivalent to an explicit direct integral of irreducible 
representations. The parameter space for the direct integral
can be taken to be a wavelet set in the sense of \cite{dl} and
\cite{dls}. In \cite{m1}, \cite{m2}, and \cite{m3}, 
wavelet sets with very simple geometric structure (a finite union
of convex sets) are constructed for increasingly more general dilation matrices. The abstract group being represented is not
a Type I group. Thus, it was interesting that its natural
representation on $L^2(\R^n)$ can be so nicely decomposed into
irreducibles. See \cite{cmo}, \cite{dut}, \cite{dhjp}, and
\cite{ds} for works partially motivated by the direct integral
decomposition obtained in \cite{lpt}. 

Although much of what we do below can be carried out in arbitrary dimensions, we restrict ourselves to the two dimensional case
where computational details are more manageable. 
We also note that the analog of simple
wavelet sets for shifts by a crystal symmetry group have
recently been constructed, see \cite{KM}, for all two dimensional crystal groups and an appropriate dilation matrix.
Besides restricting to two dimensions, we further reduce
notational details by using dilation by 3, which is compatible with all wallpaper groups, in all cases.

After introducing our notation and basic definitions in Section 2, we construct a particular semi-direct product
group, denoted $\Gamma_3\rtimes_{\vartheta}\Z$, in Section 3 and define the $3\Gamma$-wavelet 
representation of $\Gamma_3\rtimes_{\vartheta}\Z$. 
This is the unitary 
representation whose image is $\cG(A,\Gamma)$, where
the matrix $A$ dilates by 3. In Section 4, we give an
explicit construction of a family of irreducible
unitary representations of $\Gamma_3\rtimes_{\vartheta}\Z$.
In the final section, we display the $3\Gamma$-wavelet 
representation of $\Gamma_3\rtimes_{\vartheta}\Z$ as the
direct integral of the irreducible representations
constructed in Section 4.

\section{Notation and Basic Results}
We have selected notation to emphasize the central role of the natural unitary representation of the group of affine transformations of $\R^n$ and so that, as much as possible, we are
consistent with the notation of \cite{mt} and \cite{lpt}.

Let $n\in\N$ and let $\cGL_n(\R)$ denote the group of invertible linear
 transformations of $\R^n$. Let ${\rm id}$ denote the
 identity in $\cGL_n(\R)$. For any
$x\in\R^n$ and $L\in \cGL_n(\R)$, define the
transformation $[x,L]$ of $\R^n$ by
\[
[x,L]z=L(z+x),\quad\text{for all}\,\,z\in\R^n.
\]
Let ${\rm Aff}(\R^n)=\{[x,L]:x\in\R^n, L\in 
\cGL_n(\R)\}$. Under composition as group product, 
${\rm Aff}(\R^n)$ is the group of all
invertible
{\it affine transformations} of $\R^n$.
Note that, for $[x,L],[y,M]\in{\rm Aff}(\R^n)$, 
\[
[x,L]\big([y,M]z\big)=
L\big(M(z+y)+x\big)=LM\big(z+(M^{-1}x+y)\big)=
[M^{-1}x+y,LM]z,
\]
 for all $z\in\R^n$.
Thus, $[0,{\rm id}]$ is the identity in
${\rm Aff}(\R^n)$,
\[
[x,L][y,M]=[M^{-1}x+y,LM],\quad\text{and}\quad
[x,L]^{-1}=[-Lx,L^{-1}].
\]

We note that the action of ${\rm Aff}(\R^n)$ on $\R^n$ provides
a natural unitary representation of 
${\rm Aff}(\R^n)$ on $L^2(\R^n)$, which we will denote by
$R$ regardless of the dimension $n$. That is, for 
$[x,L]\in{\rm Aff}(\R^n)$ and $g\in L^2(\R^n)$,
\begin{equation}\label{nat_rep}
R[x,L]g(y)=|\det(L)|^{-1/2}g\big([x,L]^{-1}y\big))=
|\det(L)|^{-1/2}g(L^{-1}y-x),\quad\text{for all  }
y\in\R^n.
\end{equation}
We will refer to $R$ as the \emph{natural 
representation}.

Let ${\rm Trans}(\R^n)=\{[x,{\rm id}]:x\in\R^n\}$,
the normal subgroup of ${\rm Aff}(\R^n)$
consisting of pure translations.
Define $q:{\rm Aff}(\R^n)\to\cGL_n(\R)$ by
$q[x,L]=L$, for all $[x,L]\in {\rm Aff}(\R^n)$. Then
$q$ is a homomorphism onto $\cGL_n(\R$ with 
$\ker(q)={\rm Trans}(\R^n)$. We note that the restriction of the natural representation $R$ to ${\rm Trans}(\R^n)$ 
gives us the usual translation unitaries. That is, if
$x\in\R^n$, then  $R[x,{\rm id}]=T_x$, where
$T_xf(y)=f(y-x)$, for all $y\in\R^n$ and $f\in L^2(\R^n)$.

Let $\cO_n$ denote the group of orthogonal 
transformations of $\R^n$ and let 
${\rm Iso}(\R^n)$ denote the subgroup
of ${\rm Aff}(\R^n)$ consisting of all
affine transformations of the form $[x,L]$,
with $x\in\R^n$ and $L\in\cO_n$. These are the rigid
motions of $\R^n$.
An $n$-dimensional crystal group is a discrete
subgroup $\Gamma$ of ${\rm Iso}(\R^n)$ such
that the quotient space $\R^n/\Gamma$ is
compact. We refer to a $2$-dimensional crystal
group as {\it wallpaper group}. We will restrict
our attention to wallpaper groups from now on.

Let $\Gamma$ be a fixed wallpaper group. Let
$\cN={\rm Trans}(\R^2)\cap\Gamma$, the pure translations in $\Gamma$. Then $\cN$ is a 
normal subgroup of $\Gamma$. There are two
linearly independent vectors $u,v\in\R^2$
such that $\cN=\{[ju+kv,{\rm id}]:(j,k)\in\Z^2\}$.
Therefore $\cN$ is isomorphic to the lattice 
$N_\Gamma=\{ju+kv:(j,k)\in\Z^2\}$ in $\R^2$
and, hence, isomorphic to $\Z^2$. Let
\[
\cD=q(\Gamma)=\{L\in\cO_2:[x,L]\in\Gamma,\,\,\text{for
some}\,\,x\in\R^2\}.
\]
Then $\cD$ is a finite subgroup of $\cO_2$ called the
{\it point group} of $\Gamma$. 
There are 10 possibilities for $\cD$ and they
separate into two types depending on 
whether or not the group contains a reflection.
If there is no reflection, then $\cD$ is either
the trivial group or generated by a rotation
of order 2, 3, 4, or 6. 
Let $\rho(\theta)$ denote rotation through the angle $\theta$.
If $k\in\{1,2,3,4,6\}$, 
then the cyclic group of order $k$ is 
$\cC_k=\{\rho(2\pi/k)^{\,j}:0\leq j\leq k-1\}$. Note that
$\cC_1$ is the trivial group. If $\cD$
contains a reflection $S$ about some one
dimensional subspace of $\R^2$, then there
is a $k\in\{1,2,3,4,6\}$ such that
\[
\cD=\cD_k=\{\rho(2\pi/k)^{\,j}:0\leq j\leq k-1\}
\cup\{\rho(2\pi/k)^{\,j}S:0\leq j\leq k-1\}.
\]

If we denote the restriction of $q$ to $\Gamma$ by $q$ again, then $q$ 
is a homomorphism of $\Gamma$ onto
$\cD$ and $\ker(q)=\cN$. Thus $\cD$ is isomorphic
with $\Gamma/\cN$. If $[0,L]\in\Gamma$, for all
$L\in\cD$, then $\Gamma$ is called {\it 
symmorphic}. Otherwise, $\Gamma$ is 
{\it nonsymmorphic} and there is no subgroup
of $\Gamma$ with the property that $q$ restricted
to that subgroup is an isomorphism with $\cD$.  In either case, we define 
\[
\cD^0=\{L\in\cO_2:[0,L]\in\Gamma\},
\]
and note that $\Gamma$ is symmorphic exactly when $\cD^0=\cD$.

We
draw particular attention to the wallpaper group
denoted $pg$. This group is generated by $\{[u,{\rm id}],[v,{\rm id}],[(1/2)v,S]\}$, where $S\in\cO_2$ is reflection in the line determined by multiples of $v$.   The group $pg$ has a rectangular lattice \cite{mor}, so we can assume that $u$ and $v$ are along the horizontal and vertical axes respectively,  so that  
$[(1/2)v,S]$ moves up by one half of a vertical 
translation unit and reflects about the vertical 
axis. This is 
called a {\it glide-reflection}.
We have
\[
pg=\{[ju+kv,{\rm id}]:(j,k)\in\Z^2\}\bigcup
\left\{\left[ju+\left(k+\frac{1}{2}\right)v,S\right]:
(j,k)\in\Z^2\right\}.
\]
Note that $pg$ is nonsymmorphic as $[0,S]\notin pg$, and that here we have $\cD=\{\rm id, S\}$, while 
$\cD^0=\{\rm id\}$.

For a general wallpaper group $\Gamma$,
we will be interested in
\[
T_{\Gamma}=\{x\in\R^2:[x,L]\in\Gamma,
\,\,\text{for some}\,\,L\in\cD\},
\] 
and also, for each $L\in\cD,$ 
\[
T^L_{\Gamma}=\{x\in\R^2:[x,L]\in\Gamma\}.
\]

If $\Gamma$ is symmorphic,
then $T_{\Gamma}=\{ju+kv:(j,k)\in\Z^2\}$, a lattice in
$\R^2$. However,
there are four nonsymmorphic wallpaper groups. They
are $pg$, $pmg2$, $pgg2$, and $p4mg$ in a standard
notation scheme (see e.g. \cite{mor} or \cite{sch}). For each of these, extra elements are
added to the lattice to form $T_{\Gamma}$, as described in the following lemma. 

\begin{lemm}\label{T_gamma_lemma}
If $\Gamma$ is a nonsymmorphic wallpaper group and
$u,v\in\R^2$ are such that 
$\cN=\{[ju+kv,{\rm id}]:(j,k)\in\Z^2\}$, then
\begin{enumerate}
\item $u \perp v$, so we may assume $u=(1,0)$ and $v=(0,1)$.
\item There exists a $z\in\R^2$ such that 
$T^S_{\Gamma}=
\left\{ju+kv+\frac12z:(j,k)\in\Z^2\right\},$ where $S$ is reflection in the vertical axis.
If $\Gamma$ is either $pg$ or $pmg2$,
we can take $z=v$; if $\Gamma$ is either $pgg2$ or $p4mg$, we can take $z=u+v.$
\item For $L\in\cD$, $L\neq S$, $T^L_{\Gamma}=\{ju+kv:(j,k)\in\Z^2\}$
\end{enumerate}
\end{lemm}
\begin{proof}
See, e.g.\cite{mor}.
\end{proof}

Let $A\in \cGL_2(\R)$. We say $A$ is 
{\it compatible with $\Gamma$} if all eigenvalues of $A$ have
absolute value larger than 1 (that is, $A$ is a dilation matrix),
$[0,A]\Gamma[0,A]^{-1}\subseteq\Gamma$
(in which case $[0,A]\Gamma[0,A]^{-1}$ is a subgroup of $\Gamma$),
 and $\Gamma/[0,A]\Gamma[0,A]^{-1}$
is finite. Let $[x,L]\in\Gamma$.
Then
\[
[0,A][x,L][0,A]^{-1}=[x,AL][0,A^{-1}]=
[Ax,ALA^{-1}].
\]
Thus, $A$ is compatible with $\Gamma$ means 
$A\cD A^{-1}\subseteq\cD$ so, since $\cD$ is finite, 
$A\cD A^{-1}=\cD$. It also means that $AN_\Gamma\subseteq N_\Gamma\subseteq A^{-1}N_\Gamma$,
$\cap_{k=1}^\infty A^kN_\Gamma=\{0\}$, and 
$\cup_{k=1}^\infty A^{-k}N_\Gamma$ is dense in $\R^2$.
\begin{rem}
If $A\in \cGL_2(\R)$ is compatible with
$\Gamma$, then applying the natural representation, given by \eqref{nat_rep}, to
$[0,A]$ gives $R[0,A]g(y)=|\det(A)|^{-1/2}
g(A^{-1}y)$, for $y\in\R^2, g\in L^2(\R^2)$.
That is, $R[0,A]=D_{A^{-1}}$ in the notation
of \cite{lpt}. Recall, for $x\in N_\Gamma$, $R[x,{\rm id}]g(y)=
g(y-x)$, for $y\in\R^2, g\in L^2(\R^2)$. Thus,
$R(\Gamma)\cup\{R[0,A^k]:k\in\Z\}$ contains all the shifts
by vectors in the lattice 
$N_\Gamma$ and dilations by powers of $A$.
\end{rem}
\begin{definition}
If $\Gamma$ is a wallpaper group and $A\in\cGL_2(\R)$
is compatible with $\Gamma$, let
 $\cG(A,\Gamma)$ denote the smallest group of unitary operators on
$L^2(\R^2)$ containing $R[0,A]$ and the set $R(\Gamma)$.
\end{definition}
If $\cH$ is a Hilbert space, $\cB(\cH)$ denotes
the Banach algebra of bounded linear operators on $\cH$.
For $\cS\subseteq\cB(\cH)$, the \emph{commutant} of $\cS$ is
$\cS'=\{B\in\cB(\cH):BS=SB, \,\forall S\in\cS\}$. For later use, we record an observation on the commutant of $\cG(A,\Gamma)$.

\begin{prop}\label{abelian_comm}
The commutant of $\cG(A,\Gamma)$ in $\cB\big(L^2(\R^2)\big)$ is
abelian.
\end{prop}
\begin{proof}
In the notation of \cite{lpt}, $\cG(A,\Gamma)$ contains the
$\{D_A, T_v\, \big|\, v\in N_\Gamma\}$, so $\cG(A,\Gamma)'$
is contained in $\{D_A, T_v\, \big|\, v\in N_\Gamma\}'$. This latter 
set is abelian by
equation (2) in \cite{lpt}, after implementing the unitary equivalence given by the Fourier transform.
\end{proof}

\section{The $3\Gamma$-wavelet representation}
With the goal of understanding $\cG(A,\Gamma)$ better, we take
a closer look at the subgroup of ${\rm Aff}(\R^2)$ generated 
by $[0,A]$ and $\Gamma$. For each $\ell\in\Z$, $[0,A]^{\ell}\Gamma[0,A]^{-\ell}$ is a
subgroup of ${\rm Aff}(\R^2)$ and
\[
\cdots\subseteq[0,A]^{2}\Gamma[0,A]^{-2}\subseteq
[0,A]^{1}\Gamma[0,A]^{-1}\subseteq
\Gamma\subseteq
[0,A]^{-1}\Gamma[0,A]^{1}\subseteq
[0,A]^{-2}\Gamma[0,A]^{2}\subseteq\cdots.
\]
Let $\Gamma_A=
\cup_{\ell\in\Z}[0,A]^{\ell}\Gamma[0,A]^{-\ell}=
\cup_{m=1}^\infty[0,A]^{-m}\Gamma[0,A]^{m}$. Then
$\Gamma_A$ is a countable subgroup of ${\rm Aff}(\R^n)$
such that, if $[x,L]\in\Gamma_A$, then $L\in\cD$. 
 Let $\cN_A={\rm Trans}(\R^2)\cap\Gamma_A$,
the pure translations in $\Gamma_A$.  There
are two subsets of $\R^2$ that are of particular 
interest to us.
Let $T_{\Gamma_A}=
\{x\in\R^2:[x,L]\in\Gamma_A,\,\,\text{for some}\,\,
L\in\cD\}$.  Finally, let 
$N_{\Gamma_A}=\left\{x\in\R^2:[x,0]\in\cN_A\right\}=
\cup_{k=1}^\infty A^{-k}N_\Gamma$. 

\begin{prop} Let $\Gamma$ be a wallpaper group and let 
$A\in\cGL_2(\R)$ be compatible with $\Gamma$. Then
\begin{enumerate}
\item $\cN_A$ is a normal subgroup of $\Gamma_A$.
\item $q(\Gamma_A)=\cD$.
\item $\Gamma_A/\cN_A$ is isomorphic to $\cD$.
\item Both $N_{\Gamma_A}$ and $T_{\Gamma_A}$ are dense in $\R^2$.
\end{enumerate}
\end{prop}
\begin{proof} 
(1) and (2) are immediate since ${\rm Trans}(\R^2)$ is a normal subgroup of $\cGL_2(\R)$. Since $\cN_A=\{[x,L]\in\Gamma_A:
q[x,L]={\rm id}\}$ and $q$ is a homorphism onto $\cD$ when restricted to $\Gamma_A$, (3) follows. As observed above,
$N_{\Gamma_A}=\cup_{k=1}^\infty A^{-k}N_\Gamma$ is dense in 
$\R^2$. Thus $T_{\Gamma_A}$ is dense as well. 
\end{proof}

In what follows, we will restrict our attention to compatible $A$
in the
center of $\cGL_2(\R)$. This significantly simplifies
calculations. 
For any $d\in\N$, $d\geq 2$, $A=d\cdot{\rm id}$ is compatible
with each of the symmorphic wallpaper groups, because
$ALA^{-1}=L$,  and $A(ju+kv)=(dj)u+(dk)v
\in T^L_{\Gamma}$ for each $L\in\cD$
and all $ju+kv\in T_{\Gamma}$. However, as we will see in the proof of Proposition \ref{compat}, $d$ must be odd
for $A$ to be compatible with a nonsymmorphic group. 
\begin{prop}
\label{compat}
Let $A=d\cdot{\rm id}$, with $d\in\N$ odd.  
Then $A$ is compatible with all 17 of the wallpaper groups.
\end{prop}
\begin{proof}
Since for any $[x,L]\in \Gamma$ where $\Gamma$ is any wallpaper group, we have $ALA^{-1}=L$,  we need only verify $Ax\in T^L_{\Gamma}$
for all $[x,L]\in \Gamma$.  
This is clearly true if $x$ is of the form $ju+kv$, with
$(j,k)\in\Z^2$.
The other possibility for $[x,L]$ is $[ju+kv+\frac12z,S]$ with $(j,k)\in\Z^2$, where $z=v$ or $z=u+v$, and $S$ is reflection in the vertical axis.  In either case, since $d$ is odd, 
$Ax=dju+dkv+\frac d2z$ is of the form $j'u+k'v+\frac12 z$ with $(j',k')\in\Z^2$, and thus is in $T^S_{\Gamma}$.  Note that for $d$ even, $Ax$ is of the form $j'u+k'v$, with $(j',k')\in\Z^2$, which is not in $T^S_{\Gamma}$.   
\end{proof}
To further simplify, we will
work with the smallest available $d$;  
from now on $A=3\cdot{\rm id}$. There would be no meaningful
change in the following if 3 is replaced by any odd integer greater than 1. To acknowledge that $A=3\cdot{\rm id}$ from now
on $\Gamma_A, \cN_A, N_{\Gamma_A}$ and $T_{\Gamma_A}$ are written as $\Gamma_3, \cN_3, N_{\Gamma_3}$ and $T_{\Gamma_3}$, respectively.

We can use conjugation by $[0,3\cdot{\rm id}]$ to define an action 
$\vartheta$ of
$\Z$ on $\Gamma_3$. For $\ell\in\Z$,
$\vartheta_\ell$ is defined on $\Gamma_3$ by 
 $\vartheta_\ell[x,L]=
 [0,3\cdot{\rm id}]^{-\ell}[x,L][0,3\cdot{\rm id}]^\ell=
[3^{-\ell}x,L]$, for all $[x,L]\in\Gamma_3$. We then form
the semi-direct product group
\[
\Gamma_3\rtimes_{\vartheta}\Z=
\{\big([x,L],\ell\big):[x,L]\in\Gamma_3, \ell\in\Z\},
\]
equipped with group product
\begin{equation}\label{group_product}
\big([x,L],\ell\big)\big([y,M],m\big)=
\big([x,L](\vartheta_\ell[y,M]),\ell+m\big)=
\big([M^{-1}x+3^{-\ell}y,LM],\ell+m\big).
\end{equation}
Note that $\big([x,L],\ell\big)^{-1}=
\big([-3^\ell Lx,L^{-1}],-\ell\big)$, for
$\big([x,L],\ell\big)\in 
\Gamma_3\rtimes_{\vartheta}\Z$.
We will identify $\Gamma_3$ with 
$\{\big([x,L],0\big):[x,L]\in\Gamma_3\}$, a
normal subgroup of 
$\Gamma_3\rtimes_{\vartheta}\Z$. Likewise, we
identify $\cN_3$ with its copy inside
$\Gamma_3\rtimes_{\vartheta}\Z$.
\begin{prop}\label{N_A_normal}
$\cN_3$ is a normal subgroup of 
$\Gamma_3\rtimes_{\vartheta}\Z$ and
$\Gamma_3\rtimes_{\vartheta}\Z/\cN_3$ is
isomorphic to $\cD\times\Z$.
\end{prop}
\begin{proof}
Considering $\cN_3$ as a normal subgroup of 
$\Gamma_3$ for the moment we have, 
for any $\ell\in\Z$, 
$\vartheta_\ell[x,{\rm id}]=
[3^{-\ell}x,{\rm id}]\in\cN_3$, for any 
$[x,{\rm id}]\in\cN_2$. This means $\cN_3$ is
also normal in the semi-direct product
$\Gamma_3\rtimes_{\vartheta}\Z$. The
map $Q$ defined by
$Q\big([x,L],\ell\big)= (L,\ell)$ is
a homomorophism of 
$\Gamma_3\rtimes_{\vartheta}\Z$ onto 
$\cD\times\Z$ and $\ker(Q)=\cN_3$. This
shows that
$\Gamma_3\rtimes_{\vartheta}\Z/\cN_3$ is
isomorphic to $\cD\times\Z$.
\end{proof}
We will need to factor elements
of $\Gamma_3\rtimes_{\vartheta}\Z$ in a 
particular manner. Although this is just an observation, 
we state it as a lemma for future reference.
\begin{lemm}\label{factoring}
For $\big([x,L],\ell\big)\in
\Gamma_3\rtimes_{\vartheta}\Z$, we have
\[
\big([x,L],\ell\big)=\big([0,{\rm id}],\ell\big)
\big([3^\ell x,L],0\big).
\]
\end{lemm}

\begin{prop} Let $\Gamma$ be a wallpaper group 
and let $\Gamma_3$, $\cN_3$ and 
$T_{\Gamma_3}$ be as defined above. The following hold:
\begin{enumerate}
\item $\cN_3=\left\{\left[\left(\frac{j}{3^\ell}\right)u+
\left(\frac{k}{3^\ell}\right)v,{\rm id}\right]
:(j,k)\in\Z^2, \ell=0,1,2,\cdots\right\}$.
\item If $\Gamma$ is symmorphic, then 
$T_{\Gamma_3}=\{x\in\R^2:[x,{\rm id}]\in\cN_3\}$.
\item If $\Gamma$ is nonsymmorphic, then\\
$T_{\Gamma_3}=\bigcup_{\ell=0}^\infty\big(
\left\{\left(\frac{j}{3^\ell}\right)u+
\left(\frac{k}{3^\ell}\right)v:(j,k)\in\Z^2\right\}
\cup\left\{
\left(\frac{j}{3^\ell}\right)u+
\left(\frac{k}{3^\ell}\right)v+
\left(\frac{1}{2}\right)z
:(j,k)\in\Z^2\right\}\big)$, 
where $z$ is the vector identified
in Lemma \ref{T_gamma_lemma}.
\end{enumerate}
\end{prop}
\begin{proof}
(1) and (2) are clear, so we consider (3). For each integer
$\ell\geq 0$, let 
\[
T_\ell=\{x\in\R^2:[x,L]\in 
[0,3\cdot{\rm id}]^{-\ell}\Gamma[0,3\cdot{\rm id}]^\ell,\,\,\text{for some}\,\,
L\in\cD\}.
\]
We show that
\begin{equation}
\label{Tl}
T_\ell=\left\{\left(\frac{j}{3^\ell}\right)u+
\left(\frac{k}{3^\ell}\right)v,
\left(\frac{j}{3^\ell}\right)u+
\left(\frac{k}{3^\ell}\right)v+
\left(\frac{1}{2}\right)z
:(j,k)\in\Z^2\right\},
\end{equation}
 for each $\ell\geq 0$, by
induction. When $\ell=0$, $T_0=T_{\Gamma}$ and the claim
holds by the choice of $z$. Suppose the claim holds
for some $\ell\geq 0$. For any $x\in T_{\ell+1}$,
there exists $L\in\cD$ such that
$[x,L]\in [0,3\cdot{\rm id}]^{-\ell-1}\Gamma[0,3\cdot{\rm id}]^{\ell+1}$. Thus,
$[3x,L]=[0,3\cdot{\rm id}][x,L][0,3\cdot{\rm id}]^{-1}\in 
[0,3\cdot{\rm id}]^{-\ell}\Gamma[0,3\cdot{\rm id}]^\ell$. By the inductive
hypothesis, either $3x=\left(\frac{j}{3^\ell}\right)u+
\left(\frac{k}{3^\ell}\right)v$ or
 $3x=\left(\frac{j}{3^\ell}\right)u+
\left(\frac{k}{3^\ell}\right)v+\left(\frac{1}{2}\right)z$, for some $(j,k)\in \Z^2$. Thus, we have that either
$x=\left(\frac{j}{3^{\ell+1}}\right)u+
\left(\frac{k}{3^{\ell+1}}\right)v$ or
 $x=\left(\frac{j}{3^{\ell+1}}\right)u+
\left(\frac{k}{3^{\ell+1}}\right)v+
\left(\frac{1}{6}\right)z$, for some $(j,k)\in \Z^2$.
The first alternative is of the correct form. If 
 $x=\left(\frac{j}{3^{\ell+1}}\right)u+
\left(\frac{k}{3^{\ell+1}}\right)v+
\left(\frac{1}{6}\right)z$, for some $(j,k)\in \Z^2$,
we need to consider the two possibilities for $z$.
Either $z=v$ or $z=u+v$.
We note that $\frac{1}{6}=\frac{1}{2}-
\frac{3^\ell}{3^{\ell+1}}$. So
\[
x=\left(\frac{j}{3^{\ell+1}}\right)u+
\left(\frac{k-3^\ell}{3^{\ell+1}}\right)v+
\left(\frac{1}{2}\right)z,\,\,\text{if}\,\, z=v,
\]
and 
\[
x=\left(\frac{j-3^\ell}{3^{\ell+1}}\right)u+
\left(\frac{k-3^\ell}{3^{\ell+1}}\right)v+
\left(\frac{1}{2}\right)z,\,\,\text{if}\,\, z=u+v.
\]
Thus, Equation (\ref{Tl}) holds for all integers $\ell\geq 0$, and
this verifies Condition (3) of the Propostion.
\end{proof}

Although we will not use the following facts, it is interesting to
note that $x\to[x,{\rm id}]$ embeds $T_{\Gamma_3}$
as a dense subgroup of ${\rm Trans}(\R^2)$ in which
$\cN_3$ is an index two subgroup. Notice also that
$\cN_3$ is exactly the intersection of 
this larger subgroup of ${\rm Trans}(\R^2)$ with
$\Gamma_3$.

In the theory of wavelets with crystal symmetries as ``shifts''
as developed in \cite{mt} the role of the translation unitaries
is replaced by $R[x,L]$, with $[x,L]\in\Gamma$. Therefore,
the wavelet representation defined in equation (1) of \cite{lpt}
generalizes to the following map of 
$\Gamma_3\rtimes_{\vartheta}\Z$ into the unitary group of
$L^2(\R^2)$.
\begin{definition}
The {\it $3\Gamma$-wavelet representation} is the map $V$ of
$\Gamma_3\rtimes_{\vartheta}\Z$ into the group of unitary
operators on $L^2(\R^2)$ defined by, for $\big([x,L],\ell\big)\in
\Gamma_3\rtimes_{\vartheta}\Z$,
\[
V\big([x,L],\ell\big)=R[x,L]D_3^{\,\ell},
\]
where $D_3g(y)=3g(3y)$, for $y\in\R^2$ and $g\in L^2(\R^2)$.
\end{definition}
\begin{prop}\label{onto_rep}
Let $\Gamma$ be a wallpaper group and let $A=3\cdot{\rm id}$. Then the $3\Gamma$-wavelet representation is a faithful unitary
representation of $\Gamma_3\rtimes_{\vartheta}\Z$ on $L^2(\R^2)$.
Moreover, 
\[
V\big(\Gamma_3\rtimes_{\vartheta}\Z\big)=
\cG(A,\Gamma).
\]
\end{prop}
\begin{proof}
Direct computation shows that, for any $[y,M]\in \Gamma_3$,
\begin{equation}\label{commutator}
D_3R[y,M]=R[3^{-1}y,M]D_3.
\end{equation}
Using this repeatedly, we have, for 
$\big([x,L],\ell\big), \big([y,M],m\big)\in 
\Gamma_3\rtimes_{\vartheta}\Z$,
\[
V\big([x,L],\ell\big)V\big([y,M],m\big)=R[x,L]D_3^{\,\ell}
R[y,M]D_3^{\,m}=R[x,L]R[3^{-\ell}y,M]D_3^{\,\ell+m}.
\]
But $R$ is a homomorphism, so 
$R[x,L]R[3^{-\ell}y,M]=R\big([x,L][3^{-\ell}y,M]\big)=
R[M^{-1}x+3^{-\ell}y,LM]$ and, thus, 
$V\big([x,L],\ell\big)V\big([y,M],m\big)=
R[M^{-1}x+3^{-\ell}y,LM]D_3^{\,\ell+m}$. Using \eqref{group_product}, we see that $V$ is a homomorphism of
$\Gamma_3\rtimes_{\vartheta}\Z$ into the unitary group of
$L^2(\R^2)$. It is clear that the image of $V$ is the 
smallest group of unitary operators on $L^2(\R^2)$ containing
$R[0,A]=D_3^{-1}$ and the set $R(\Gamma)$. That is, the image of
$\Gamma_3\rtimes_{\vartheta}\Z$ under $V$ is $\cG(A,\Gamma)$.
Note that, for $\big([x,L],\ell\big)\in
\Gamma_3\rtimes_{\vartheta}\Z$, if $R[x,L]D_3^{\,\ell}$ is the
identity operator, then $R[x,L]=D_3^{\,-\ell}$, from which one can 
verify that $x=0$, $L={\rm id}$, and $\ell=0$. This implies
$V$ is faithful.
\end{proof}

It will be useful to convert $V$ to an equivalent
representation $\widehat{V}$ using the Fourier
transform. We use the following form of the
Fourier transform. For $g\in L^1(\R^2)$,
\[
\cF(g)(\omega)=\widehat{g}(\omega)=\int_{\R^2}g(x)e^{-2\pi i\langle x, \omega\rangle}dx,
\,\,\text{for all}\,\,\omega\in\R^2.
\]
For $\big([x,L],\ell\big)\in 
\Gamma_3\rtimes_{\vartheta}\Z$, let 
$\widehat{V}\big([x,L],\ell\big)=\cF V\big([x,L],\ell\big)
\cF^{-1}$. A direct computation provides an explicit formula for
$\widehat{V}$.
\begin{prop}\label{hat_rho}
For any $\big([x,L],\ell\big)\in \Gamma_3\rtimes_{\vartheta}\Z$
and any $h\in L^2(\R^2)$,
\[
\widehat{V}\big([x,L],\ell\big)h(\omega)=
3^{-\ell}e^{-2\pi i\langle x,L^{-1}\omega\rangle}
h(3^{-\ell}L^{-1}\omega),\,\,\text{for all}\,\,\omega\in\R^2.
\]
\end{prop}
Our primary goal is to decompose $\widehat V$ as a direct integral of irreducible representations that we describe in the next section.   
%%%%%%%%%%%%%%%%%%%%%%%%

\section{A family of irreducible representations of $\Gamma_3\rtimes_{\vartheta}\Z$}
The components in our decomposition of the $3\Gamma$-wavelet representation will be certain irreducible
representations of $\Gamma_3\rtimes_{\vartheta}\Z$ each of which is induced from a
character of $\cN_3$.
 
The normal subgroup
$\cN_3$ is a countable discrete abelian group.
Its dual, $\widehat{\cN_3}$, is a compact abelian
group. There is a distinguished subset of $\widehat{\cN_3}$ consisting of restrictions of continuous characters of $\R^2$ to $N_3$, then moved to $\cN_3$. For
each $\omega\in\R^2$, define
$\chi_\omega:\cN_3\to\T$ by
\[
\chi_\omega\big([x,{\rm id}],0\big)=
e^{-2\pi i\langle x,\omega\rangle},\,\quad\text{for all}\,\, 
\big([x,{\rm id}],0\big)\in\cN_3.
\]
Because $N_3$ is dense in $\R^2$, $\chi_\omega=\chi_{\omega'}$
if and only if $\omega=\omega'$, for $\omega,\omega'\in\R^2$.
\begin{prop}\label{Omega_dense}
The map $\omega\to\chi_\omega$ is a continuous one-to-one 
homomorphism
of $\R^2$ onto a dense subgroup of $\widehat{\cN_3}$.
\end{prop}
\begin{proof}
The facts that $\omega\to\chi_\omega$ is
one-to-one and a homomorphism are
 obvious.
Since $\cN_3$ is being considered with the discrete topology,
the topology of $\widehat{\cN_3}$ is the topology of
pointwise convergence, so continuity is easy. Let
$\Omega=\{\chi_\omega:\omega\in\R^2\}$, a subgroup of 
$\widehat{\cN_3}$. For any $[x,{\rm id}]\in\cN_3$, 
if $[x,{\rm id}]\neq[0,{\rm id}]$, then
there exists $\omega\in\R^2$ such that 
$\chi_\omega[x,{\rm id}]\neq 1$.
Thus, the annihilator of $\Omega$ in $\cN_3$ is 
$\{[0,{\rm id}]\}$, which
means that the double annihilator is $\widehat{\cN_3}$. But
the double annihilator is $\overline{\Omega}$, see \cite{hr}. This shows that
$\Omega$ is dense in $\widehat{\cN_3}$.
\end{proof}
Thus, we can think of $\Omega$ as a copy of $\R^2$ equipped
with a weaker topology sitting densely in $\widehat{\cN_3}$. 

We will now induce the characters in $\Omega$ to 
$\Gamma_3\rtimes_{\vartheta}\Z$.  
There are several different versions of induced representations that can all be shown to be unitarily equivalent to one another.  Here we use the version given in \cite{KT} Chapter 2 and in \cite{fol} Chapter 6.1, Remark 2, p. 155.  The basis for these descriptions appeared in  \cite{mac}.

This general definition for induced representations applies to a representation $\pi$, acting in the Hilbert space $\mathcal H_{\pi}$, of a closed subgroup $H$ of a locally compact group $G$.  In general, the definition involves the Radon-Nikodym derivative $\lambda$ for a quasi-invariant measure under the left action of $G$ on $G/H$.  The induced representation acts in a Hilbert space of square integrable functions $f:G\mapsto\mathcal H_{\pi}$ satisfying $f(xh)=\pi(h^{-1})(f(x)),\;\forall h\in H,\;\forall g\in G.$ by $$U_{\gamma}^{\pi}(x)(f)(y)=\sqrt{\lambda(x^{-1}, yH)}f(x^{-1}y).$$
We are interested in the following special case:

\begin{definition}
For each $\chi_\omega\in\Omega$, let $U^\omega=
{\rm Ind}_{\cN_A}^{\Gamma_A\rtimes_{\vartheta}\Z}\chi_\omega$, the representation of
$\Gamma_3\rtimes_{\vartheta}\Z$ induced from the representation $\chi_\omega$ of
$\cN_3$.
\end{definition}

Since $\Gamma_3\rtimes_{\vartheta}\Z$ is discrete $\cN_3$ is an
open subgroup. Also, the $\chi_\omega$ are one dimensional
representations of $\cN_3$. This makes inducing easier with
details for inducing from an open subgroup worked out in 
\cite{KT}, Section 2.1.
We will provide an explicit formula for $U^\omega$ below.  However, both to develop the form we need, and to determine which $U^\omega$ are irreducible and
when two of them are equivalent, we first need to understand the action of $\Gamma_3\rtimes_{\vartheta}\Z/\cN_3=\cD\times\Z$ on 
$\widehat{\cN_3}$. 

Elements of $\Gamma_3\rtimes_{\vartheta}\Z$ act on the normal subgroup $\cN_3$ by
conjugation. That is, for 
$\big([x,L],\ell\big)
\in\Gamma_3\rtimes_{\vartheta}\Z$ and 
$\big([y,{\rm id}],0\big)\in\cN_3$, 
\[
\big([x,L],\ell\big)\cdot\big([y,{\rm id}],0\big)=
\big([x,L],\ell\big)\big([y,{\rm id}],0\big)\big([x,L],\ell\big)^{-1}
\]
\[
=\big([x+3^{-\ell}y,L],\ell\big)
\big([-3^\ell Lx,L^{-1}],-\ell\big)=
\big([3^{-\ell}Ly,{\rm id}],0\big).
\]
As expected, this action only depends on the coset of $\cN_3$
containing $\big([x,L],\ell\big)$. Thus, it
is actually an action of $\cD\times\Z$.
We write $(L,\ell)\cdot[y,{\rm id}]=
[3^{-\ell}Ly,{\rm id}]$, for each 
$(L,\ell)\in\cD\times\Z$ 
and $\big([y,{\rm id}],0\big)\in\cN_3$.
This then determines an action of $\cD\times\Z$ on 
$\widehat{\cN_3}$. For
$(L,\ell)\in\cD\times\Z$ and
 $\chi\in\widehat{\cN_3}$, define 
 $(L,\ell)\cdot\chi\in\widehat{\cN_3}$
by
\[
((L,\ell)\cdot\chi)\big([y,{\rm id}],0\big)=
\chi\big((L^{-1},-\ell)\cdot([y,{\rm id}],0)\big)
=\chi\big([3^\ell L^{-1}y,{\rm id}],0\big),
\,\,\text{for all}\\,\,\big([y,{\rm id}],0\big)
\in\cN_3.
\]
If $\omega\in\R^2$, then 
$(L,\ell)\cdot\chi_\omega=
\chi_{(3^\ell L^{-1})^t\omega}$. Note that 
$(3^\ell L^{-1})^t=3^\ell L$, since 
$\cD\subseteq\cO_2$. Thus,
$(L,\ell)\cdot\chi_\omega=
\chi_{3^\ell L\omega}$, 
for $(L,\ell)\in\cD\times\Z$, and $\Omega$
is invariant under the action of $\cD\times\Z$ on $\widehat{\cN_3}$.
To understand the orbit structure in $\Omega$ under the action of 
$\cD\times\Z$, it suffices to describe the 
orbit structure in $\R^2$
under the action of $\cD\times\Z$. For each 
$\omega\in\R^2$, let $\cD(\omega)=
\{L\omega:L\in\cD\}$,
$(\cD\times\Z)(\omega)=
\{3^\ell L\omega:(L,\ell)\in\cD\times\Z\}$ and 
$\cD_\omega=\{L\in\cD:L\omega=\omega\}$. We call $\cD(\omega)$ a {\it $\cD$-orbit},
$(\cD\times\Z)(\omega)$ a
{\it $(\cD\times\Z)$-orbit} and $\cD_\omega$ the
{\it stability subgroup} of $\omega$ in $\cD$.
Note that the stability subgroup of $\omega$
in $\cD\times\Z$ is 
$\{(L,0):L\in\cD_\omega\}$.
The following proposition
is a result of  Theorem 2.6 and Proposition 2.8 of \cite{KT}.

\begin{prop}\label{irr_equiv}
Let $\omega,\omega'\in\R^2$. 
Then $U^\omega$ is
irreducible if and only if $\cD_\omega=
\{{\rm id}\}$ and $U^{\omega'}$ is
equivalent to $U^\omega$ if and only if 
$(\cD\times\Z)(\omega')=(\cD\times\Z)(\omega)$. 
\end{prop}
For any $X\subseteq\R^2$ and $(L,\ell)\in\cD\times\Z$,
let $3^\ell LX=\{3^\ell L\omega:\omega\in X\}$.

\begin{definition}\label{weakAD}
A subset $X$ of $\R^2$ is called a 
{\it weak $3\cD$-cross-section} if \\
\indent\indent (a) $X$ is Borel.\\
\indent\indent (b) For $(L,\ell), (M,m)\in\cD\times\Z$,
$(L,\ell)\neq(M,m)$ implies 
$\big(3^\ell LX\big)\cap\big(3^m MX\big)=\emptyset$.\\
\indent\indent (c) $\cup_{(L,\ell)\in\cD\times\Z}3^\ell LX$ is
dense in $\R^2$.
\end{definition}

\begin{prop}\label{X}
Let $\cD$ be the point group of 
a wallpaper group $\Gamma$. Then a weak $3\cD$-cross-section
exists.
\end{prop}
\begin{proof}
If $\cD=\cC_k$, for some $k\in\{1,2,3,4,6\}$,
then $\cD_\omega$ is trivial, for all
$\omega\neq 0$. Let $Y=\{(r,0):r\in\R,r>0\}$
and $Z=\cup\{\rho(\theta)\omega:\omega\in Y,
0<\theta<\frac{2\pi}{k}\}$. Then, for each
$\omega\in\R^2$, $\omega\neq 0$, $\cD(\omega)
\cap(Y\cup Z)$ is a singleton. If $\cD$ contains a
reflection, then, with $Y$ as just defined,
there exist $0\leq\theta_1<\theta_2<2\pi$
such that $\{\rho(\theta)\omega:\omega\in Y,
\theta_1\leq\theta\leq\theta_2\}$ contains
exactly one member from each nonzero $\cD$-orbit
in $\R^2$. Moreover, if 
$Z=\{\rho(\theta)\omega:\omega\in Y,
\theta_1<\theta<\theta_2\}$, then $\cD_\omega$
is trivial for each $\omega\in Z$. In either case,
$Z$ is open, $\cup{L\in\cD}LZ$ is dense in $\R^2$ and
$LZ\cap MZ=\emptyset$, if $L\neq M$. Moreover, $\omega
\in Z$ implies $r\omega\in Z$, for all $r>0$.
If $\|\cdot\|$ denotes the Euclidean norm on $\R^2$, let
\[
X=\{\omega\in Z:1\leq\|\omega\|<3\}.
\]
Then $X$ is Borel, $3^\ell X\cap 3^mX=\emptyset$, for $\ell\neq m$,
and $\cup_{\ell\in\Z}3^\ell X$ is dense in $Z$. Therefore,
$X$ is a weak $3\cD$-cross-section.
\end{proof}
One interesting source of weak $3\cD$-cross-sections are 
$3\Gamma$-wavelet sets, that is, Borel sets $W\subseteq\R^2$ such that the characteristic function ${\bf 1}_W$ is the Fourier transform of an $A\Gamma$-wavelet, where $A=3\cdot{\rm id}$. 
To be the Fourier transform of an $A\Gamma$-wavelet, this characteristic function must be orthogonal to its dilates by
nontrivial powers of $A$ as well as to its transformations by elements of $\mathcal F\Gamma\mathcal F^{-1}$.  The characteristic function must also have the property that dilates of these transformations form an orthonormal basis for $L^2(\mathbb R^2)$  These properties are easily seen to imply conditions (b) and (c) of Definition \ref{weakAD}.  This is discussed further in \cite{KM}, where $A\Gamma$-wavelet sets are shown to exist for all wallpaper groups and all integer dilations.

Let us briefly recall the concept of weak equivalence for sets of unitary representations. Let $G$ be a locally compact group and let $C^*(G)$
denote the group $C^*$-algebra of $G$. See Section 7.1 of \cite{fol}
for basic information on $C^*(G)$. Any unitary representation 
$\pi$ of
$G$ determines a unique nondegenerate $*$-representation, also
denoted $\pi$, of $C^*(G)$. This correspondence of representations
preserves irreducibility and equivalence. For a unitary
representation $\pi$ of $G$, ${\rm Ker}(\pi)$ denotes the kernel
of $\pi$ as a nondegenerate $*$-representation of $C^*(G)$,
a closed $*$-ideal in $C^*(G)$.

\begin{definition}
Let $\cS$ and $\cT$ be sets of unitary representations of a
locally compact group $G$. We say that $\cS$ and $\cT$ are
\emph{weakly equivalent} if 
\[
\cap\{{\rm Ker}(\sigma):\sigma\in\cS\}=
\cap\{{\rm Ker}(\tau):\tau\in\cT\}.
\]
If $\pi$ is a single unitary representation of $G$ and $\cS$ is
weakly equivalent to $\{\pi\}$, then we simply say $\cS$ is
weakly equivalent to $\pi$.
\end{definition}

\begin{thm}
\label{weakeq}
Let $\Gamma$ be a wallpaper group with point group $\cD$. If $X$ is a weak $3\cD$-cross-section,
then $\{U^\omega:\omega\in X\}$ is weakly equivalent to the
left regular representation of $\Gamma_3\rtimes_{\vartheta}\Z$.
\end{thm}
\begin{proof}
Recall from Proposition \ref{Omega_dense} that 
$\Omega=\{\chi_\omega:\omega\in\R^2\}$ is dense in
$\widehat{\cN_3}$. Thus, the set 
$\Omega_X=
\{\chi_\omega:\omega\in\cup_{(L,\ell)\in\cD\times\Z}3^\ell LX\}$
is dense in $\widehat{\cN_3}$. 

This implies
that the left regular representation of
$\cN_3$ is weakly equivalent to
$\Omega_X$.
Now apply Corollary 5.41 of \cite{KT} to conclude that the left regular representation of 
$\Gamma_3\rtimes_{\vartheta}\Z$ is weakly equivalent to
$\{U^\omega:\omega\in 
\cup_{(L,\ell)\in\cD\times\Z}3^\ell LX\}$.
But, for every $\omega\in 
\cup_{(L,\ell)\in\cD\times\Z}3^\ell LX$, there exists an
$\omega'\in X$ so that $U^{\omega'}\sim U^\omega$. This means that
the left regular representation of 
$\Gamma_3\rtimes_{\vartheta}\Z$ is weakly equivalent to
$\{U^\omega:\omega\in X\}$.
\end{proof}

\begin{rem}\label{fath_fam}
Since $\Gamma_3\rtimes_{\vartheta}\Z$ is amenable $($see, 
for example, \cite{pat}$)$,
the left regular representation is a faithful representation of
$C^*(\Gamma_3\rtimes_{\vartheta}\Z)$ by Hulanicki's Theorem \cite{hul}. Therefore, if $X$ is a weak $3\cD$-cross-section,
then $\{U^\omega:\omega\in X\}$ is a faithful family of
irreducible representations of 
$C^*(\Gamma_3\rtimes_{\vartheta}\Z)$.
\end{rem}

In order to get an explicit expression for
the $U^\omega$, we will need to fix a 
{\it section} from the quotient group $\cD\times\Z$
into $\Gamma_3\rtimes_{\vartheta}\Z$. 
That is, we fix a map $\gamma:\cD\times\Z\to\Gamma_3\rtimes_{\vartheta}\Z$,
satisfying $(Q\circ\gamma)(L,\ell)=(L,\ell)$ as follows:
\begin{definition}\label{crosssection}
For $(L,\ell)\in\cD\times\Z$, let
\[
\gamma (L,l)=\left\{\begin{array}{ll}([0,L],l)& L\in \mathcal D^0\\
([\frac{3^{-l}}2z,L],l)&L\notin\mathcal D^0\\
\end{array}\right.,
\]
 where $z=v$ if $\Gamma$ is either $pg$ or $pmg2$ and $z=u+v$ if $\Gamma$ is either $pgg2$ or $p4mg$.
\end{definition}
This is legitimate since $[\frac{3^{-l}}2z,L]=
[0,3\cdot{\rm id}]^{-l}[\frac12z,L][0,3\cdot{\rm id}]^l\in\Gamma_3$.  

Each member
of $\Gamma_3\rtimes_{\vartheta}\Z$ can be uniquely written in the form
$\gamma(L,\ell)\big([x,{\rm id}],0\big)$ with 
$(L,\ell)\in \cD\times\Z$ and 
$\big([x,{\rm id}],0\big)\in\cN_3$. 
Indeed, for $\big([y,M],m\big)\in\Gamma_3\rtimes_{\vartheta}\Z$,
\[
\big([y,M],m\big)]=
\gamma(M,m)\big(\gamma(M,m)^{-1}\big([y,M],m\big)\big)\quad
\text{and}\quad \gamma(M,m)^{-1}\big([y,M],m\big)\in\cN_3.
\]

Since $\Gamma_3\rtimes_{\vartheta}\Z$ is discrete,
inducing a one dimensional representation, such as
$\chi_\omega$, for $\omega\in \R^2$ from $\cN_3$ up to 
$\Gamma_3\rtimes_{\vartheta}\Z$ takes
a relatively simple form (see Section 2.1 of \cite{KT}). Let
\[
\cH_{_{U^\omega}}=\{\xi:\Gamma_3\rtimes_{\vartheta}\Z\to\C
\,\,\text{satisfying (a) and (b)}\},\,\,\text{where}
\]
\indent (a) 
$\xi\big(\gamma(L,\ell)\big([x,{\rm id}],0\big)\big)=
\chi_\omega\big([-x,{\rm id}],0\big)\xi\big(\gamma(L,\ell)\big)$, for all
$(L,\ell)\in\cD\times\Z$, $\big([x,{\rm id}],0\big)\in\cN_3$.\\
\indent (b) $\sum_{(L,\ell)\in\cD\times\Z}
|\xi\big(\gamma(L,\ell)\big)|^2<\infty$.

We equip $\cH_{_{U^\omega}}$ with the inner product given by
$\langle\xi,\eta\rangle=
\sum_{(L,\ell)\in\cD\times\Z}\xi\big(\gamma(L,\ell)\big)
\overline{\eta\big(\gamma(L,\ell)\big)}$,
for $\xi,\eta\in\cH_{_{U^\omega}}$. The induced representation,
$U^\omega$ is realized on $\cH_{_{U^\omega}}$. For 
$\big([x,L],\ell\big)\in\Gamma_3\rtimes_{\vartheta}\Z$, 
$U^\omega\big([x,L],\ell\big)$ is the unitary operator
on $\cH_{_{U^\omega}}$ defined by
\[
U^\omega\big([x,L],\ell\big)\xi\big([y,M],m\big)=
\xi\big(\big([x,L],\ell\big)^{-1}\big([y,M],m\big)\big),
\] for all $\big([y,M],m\big)\in\Gamma_3\rtimes_{\vartheta}\Z$,
$\xi\in\cH_{_{U^\omega}}$. It is often useful to work with
a representation that is unitarily equivalent to 
$U^\omega$ obtained by noticing that 
$W:\ell^2(\cD\times\Z)\to\cH_{_{U^\omega}}$ given by
$Wf\big(\gamma(L,\ell)\big([x,{\rm id}],0\big)\big)=
\chi_\omega(-x)f(L,\ell)$, for
$\gamma(L,\ell)\big([x,{\rm id}],0\big)\in\Gamma_3\rtimes_{\vartheta}\Z$, 
$f\in\ell^2(\cD\times\Z)$, is
a unitary map of $\ell^2(\cD\times\Z)$ onto $\cH_{_{U^\omega}}$. Define
\[
\sigma_\omega\big([x,L],\ell\big)=
W^{-1}U^\omega\big([x,L],\ell\big)W,\quad
\text{for all}\,\,\big([x,L],\ell\big)\in\Gamma_3\rtimes_{\vartheta}\Z.
\]
Although $\sigma_\omega$ depends on the 
section $\gamma$,
we suppress its role in the notation. We note that
$W^{-1}\xi(L,\ell)=\xi\big(\gamma(L,\ell)\big)$, for 
$(L,\ell)\in\cD\times\Z$, 
$\xi\in\cH_{_{U^\omega}}$.
\begin{prop}\label{explicit_induced}
For $\omega\in\R^2$, $\sigma_\omega$ is given by
\[
\sigma_\omega\big([x,L],\ell\big)f(M,m)=
\chi_\omega\big(\gamma(M,m)^{-1}\big([x,L],\ell\big)
\gamma(L^{-1}M,m-\ell)\big)
f(L^{-1}M,m-\ell),
\]
for $(M,m)\in\cD\times\Z$, 
$f\in\ell^2(\cD\times\Z)$, and 
$\big([x,L],\ell\big)\in\Gamma_3\rtimes_{\vartheta}\Z$.
\end{prop}
\begin{proof}
Fix $(M,m)\in\cD\times\Z$ and 
$\big([x,L],\ell\big)\in\Gamma_3\rtimes_{\vartheta}\Z$. Then
\[
\big([x,L],\ell\big)^{-1}\gamma(M,m)=
\gamma(L^{-1}M,m-\ell)\big(\gamma(L^{-1}M,m-\ell)^{-1}
\big([x,L],\ell\big)^{-1}\gamma(M,m)\big).
\]
Observe that $Q\big(\gamma(L^{-1}M,m-\ell)^{-1}
\big([x,L],\ell\big)^{-1}\gamma(M,m)\big)=
(L^{-1}M,m-\ell)^{-1}(L^{-1},-\ell)(M,m)=({\rm id},0)$,
since $Q$ is a homomorphism of $\Gamma_3\rtimes_{\vartheta}\Z$ onto 
$\cD\times\Z$. Thus
\[
\gamma(L^{-1}M,m-\ell)^{-1}
\big([x,L],\ell\big)^{-1}\gamma(M,m)\in\cN_3.
\]
Therefore, for any $f\in\ell^2(\cD\times\Z)$,
\begin{eqnarray*}
\sigma_\omega\big([x,L],\ell\big)f(M,m)&=&
W^{-1}U^\omega\big([x,L],\ell\big)Wf(M,m)=
U^\omega\big([x,L],\ell\big)Wf\big(\gamma(M,m)\big)\\
&=& Wf(\big([x,L],\ell\big)^{-1}\gamma(M,m)\big)\\
&=&
Wf\left(\gamma(L^{-1}M,m-\ell)\big(\gamma(L^{-1}M,m-\ell)^{-1}
\big([x,L],\ell\big)^{-1}\gamma(M,m)\big)\right)\\
&=&\chi_\omega\big(\gamma(M,m)^{-1}\big([x,L],\ell\big)
\gamma(L^{-1}M,m-\ell)\big)
Wf\left(\gamma(L^{-1}M,m-\ell)\right)\\
&=&\chi_\omega\big(\gamma(M,m)^{-1}\big([x,L],\ell\big)
\gamma(L^{-1}M,m-\ell)\big)f(L^{-1}M,m-\ell),
\end{eqnarray*}
as asserted.
\end{proof}
Using Definition \ref{crosssection}, and the fact that $\chi_\omega\big([y,{\rm id}],0\big)=e^{-2\pi i\langle y,\omega\rangle}$ for
$\big([y,{\rm id}],0\big)\in\cN_3$, the formula for $\sigma_\omega$ simplifies as follows:

\begin{cor}
\label{finalinduced}
If $\Gamma$ is a wallpaper group and $\omega\in\R^2$,
then
\[
\sigma_\omega\big([x,L],\ell\big)f(M,m)=
\left\{\begin{array}{ll}e^{-2\pi i\langle x, 3^mL^{-1}M\omega\rangle}f((L^{-1}M,m-\ell))&L\in\mathcal D^0\\
e^{-\pi i\langle z,\omega\rangle} e^{-2\pi i\langle x,3^mL^{-1}M\omega\rangle}f((L^{-1}M,m-\ell))&L\notin\mathcal D^0,M\in\mathcal D^0\\
e^{\pi i\langle z,\omega\rangle} e^{-2\pi i\langle x,3^mL^{-1}M\omega\rangle-\frac12z}f((L^{-1}M,m-\ell))&L\notin\mathcal D^0, M\notin \mathcal D^0\\
 \end{array}\right.,
\]
for $(M,m)\in\cD\times\Z$, 
$f\in\ell^2(\cD\times\Z)$, and 
$\big([x,L],\ell\big)\in\Gamma_3\rtimes_{\vartheta}\Z$.
\end{cor}

Now, suppose that $X$ is a weak $3\cD$-cross-section. Then
$\cup_{(L,\ell)\in\cD\times\Z}3^\ell LX$ is dense in $\R^2$.
Moreover, distinct points in $X$ lie in distinct $\cD\times\Z$ orbits and the orbit of each point in $X$ is free.
In light of
Proposition \ref{irr_equiv}, $\{\sigma_\omega:\omega\in X\}$
consists of inequivalent irreducible representations of
$\Gamma_3\rtimes_{\vartheta}\Z$ and Theorem \ref{weakeq} says there are
enough of them to be weakly equivalent to the left regular
representation. We now form the direct integral of the 
$\sigma_\omega$ with respect to Lebesgue measure of $\R^2$ 
restricted to $X$. For a comprehensive treatment of the
the theory of direct integrals see
\cite{ni} or Chapter 14 of \cite{KR}. Also, Section 7.4 of \cite{fol} provides details
for direct integrals of unitary representations.

Each $\sigma_\omega$, for $\omega\in X$, acts on
the same Hilbert space, $\ell^2(\cD\times\Z)$. This makes it
relatively easy to describe the Hilbert space of their
direct integral with respect to Lebesgue measure restricted to
$X$. First, we note that a function
$F:X\to\ell^2(\cD\times\Z)$ is called \emph{measurable}, in
this context, if
$\omega\to\langle F(\omega),\eta\rangle$ is Borel measurable on
$X$, for each $\eta\in\ell^2(\cD\times\Z)$. Then
\[
L^2\big(X,\ell^2(\cD\times\Z)\big)=
\left\{F:X\to\ell^2(\cD\times\Z)\,\big|\, F\,\,\text{is measurable
and } \int_X\|F(\omega)\|^2d\omega<\infty\right\}
\]
is the Hilbert space we need.
By definition, the direct integral representation, denoted
$\int^{\oplus}_X\sigma_\omega\,d\omega$, acts 
 as follows: For
$\big([x,L],\ell\big)\in\Gamma_3\rtimes_{\vartheta}\Z$ and $F\in 
L^2\big(X,\ell^2(\cD\times\Z)\big)$,
\begin{equation}\label{dir_int}
\left[\left(\int^{\oplus}_X\sigma_\omega\,d\omega\right)
\big([x,L],\ell\big)F\right](\omega')=
\sigma_{\omega'}\big([x,L],\ell\big)F(\omega'),\quad\text{for a.e.}
\,\,\omega'\in X.
\end{equation}

%%%%%%%%%%%%%%%%%%%%%%%%%

\section{A decomposition of the wavelet representation}
\label{maintheorem}
To decompose the wavelet representation in terms of the family of irreducibles developed in the previous section, we need to conjugate $\widehat V$ by a map that breaks $\mathbb R^2$ up into a product $X\times\cD\times\mathbb Z$, where $X$ is a weak $3\cD-$cross-section.  Because of the extra factor in the induced representations for group elements involving a glide, this conjugating map will include the twist $c:X\times \mathcal D\mapsto \mathbb C$, defined by 
\begin{equation}
\label{twist}
c(\omega,L)=\left\{\begin{array}{ll}e^{-\frac{\pi i\langle z,\omega\rangle}2}&\mbox{ if } L\in \mathcal D^0\\
e^{\frac{\pi i\langle z,\omega\rangle}2}&\mbox{ if } L\notin \mathcal D^0\\
\end{array}\right.,
\end{equation}
where $z$ is the vector identified
in Lemma \ref{T_gamma_lemma}.

Given a weak $3\cD-$cross-section $X$, define the map $\rho:L^2(\mathbb R^2)\mapsto L^2(X\times \cD\times\mathbb Z)$
by
\[
[\rho(\phi)](\omega,M,j) = 3^jc(\omega,M)\phi(3^jM(\omega)),
\]
for $\omega\in X$, $j\in\mathbb Z$, and $M\in\mathcal D$.   Then $\rho$ is a Hilbert space isomorphism whose inverse is given by
\begin{equation}
\label{rhoinv}
[\rho^{-1}(f)](\xi)
= \sum_k 3^{-k}\sum_{M'\in\mathcal D}c(-3^{-k}M'^{-1}\xi,M'){\bf 1}_{(M'(X))}(3^{-k}\xi) [f(3^{-k}M'^{-1}\xi,M',k)],
\end{equation}
for $\xi\in\R^2$ and $f\in L^2(X\times \cD\times\mathbb Z)$.
If $\widetilde V$ is defined by $\widetilde V
\big([x,L],\ell\big)=\rho\,\widehat V\big([x,L],\ell\big) \rho^{-1},$ for all $\big([x,L],\ell\big)\in \Gamma_3\rtimes_{\vartheta}\Z$,
we get a representation whose detailed action is given in
the following proposition.
\begin{prop}\label{tilde_V}
For $\big([x,L],\ell\big)\in \Gamma_3\rtimes_{\vartheta}\Z$ and
$f\in L^2(X\times \cD\times\mathbb Z)$,
\begin{align*} 
[\widetilde V&\big([x,L],\ell\big)(f)](\omega,M,j)\\
&=\left\{\begin{array}{ll}e^{-2\pi i\langle x,3^jL^{-1}M(\omega)\rangle}f(\omega,L^{-1}M,j-\ell)&L\in\mathcal D^0\\
e^{-\pi i\langle z,\omega\rangle}e^{-2\pi i\langle x,3^jL^{-1}M(\omega)\rangle}f(\omega,L^{-1}M,j-\ell)]& L\notin\mathcal D^0,\; M\in\cD^0\\
e^{\pi i\langle z,\omega\rangle}e^{-2\pi i\langle x,3^jL^{-1}M(\omega)\rangle}f(\omega,L^{-1}M,j-\ell)& L\notin\mathcal D^0,\; M\notin\cD^0\\
\end{array}\right.,
\end{align*}
for a.e. $(\omega,M,j)\in X\times\cD\times\Z$.
\end{prop}
\begin{proof} We simply compute:
\begin{align*}
[\widetilde V&\big([x,L],\ell\big)(f)](\omega,M,j)\\
&= 3^{j}c(\omega,M)[\widehat V([x,L],\ell)(\rho^{-1}(f))](3^jM(\omega))\\
&= 3^{(j-\ell)}c(\omega,M)e^{-2\pi i\langle x,3^jL^{-1}M(\omega)\rangle}[\rho^{-1}(f)](3^{j-\ell}L^{-1}M(\omega))\\
&=3^{(j-\ell)}c(\omega,M)e^{-2\pi i\langle x,3^jL^{-1}M(\omega)\rangle}3^{-(j-\ell)}c(-\omega,L^{-1}M))f(\omega,L^{-1}M,j-\ell)\\
&=\left\{\begin{array}{ll}e^{-2\pi i\langle x,3^jL^{-1}M(\omega)\rangle}f(\omega,L^{-1}M,j-\ell)&L\in\mathcal D^0\\
e^{-\pi i\langle z,\omega\rangle}e^{-2\pi i\langle x,3^jL^{-1}M(\omega)\rangle}f(\omega,L^{-1}M,j-\ell) & L\notin\mathcal D^0,\; M\in\cD^0\\
e^{\pi i\langle z,\omega\rangle}e^{-2\pi i\langle x,3^jL^{-1}M(\omega)\rangle}f(\omega,L^{-1}M,j-\ell) & L\notin\mathcal D^0,\; M\notin\cD^0\\
 \end{array}\right..\\
\end{align*}
Here the second to last step follows from Equation (\ref{rhoinv}) since the only nonzero summand of $\rho^{-1}(f)(3^{j-l}L^{-1}M(\omega))$ there is for $k=j-\ell$ and $M'=L^{-1}M$. The last step fills in the definition of $c$ from Equation(\ref{twist}).
\end{proof}
Now we are ready for the main theorem of the paper.

\begin{thm}\label{final}
Let $\Gamma$ be a wallpaper group. The $3\Gamma$-wavelet
representation $V$ of $\Gamma_3\rtimes_{\vartheta}\Z$ is
equivalent to a direct integral of irreducible representations
induced from characters of the normal abelian subgroup $\cN_3$.
\end{thm}

\begin{proof}
Since $\widetilde V$ is equivalent to $V$, 
it suffices to show that the representation $\widetilde V$ is equivalent to  
$\int^{\oplus}_X\sigma_\omega\,d\omega$.  We note that the map
$W:L^2\big(X,\ell^2(\cD\times\Z)\big)\to L^2(X\times\cD\times\Z)$
given by $WF(\omega,L,\ell)=\big(F(\omega)\big)(L,\ell)$, for
all $(\omega,L,\ell)\in X\times\cD\times\Z$ and $F\in 
L^2\big(X,\ell^2(\cD\times\Z)\big)$, is a Hilbert space isomorphism,
which is easily checked.  Using \eqref{dir_int} and the explicit
formulas in Corollary \ref{finalinduced} and 
Proposition \ref{tilde_V}, one verifies directly that
\[
W\left[\left(\int^{\oplus}_X\sigma_\omega d\omega\right)
\big([x,L],\ell\big)\right]W^{-1}=\widetilde{V}\big([x,L],\ell\big),
\]
for all $\big([x,L],\ell\big)\in\Gamma_3\rtimes_{\vartheta}\Z$.  
This completes the proof.
\end{proof}

We refer to Chapter 6 of \cite{KR} for the definition of 
Type I von Neumann algebras. 
If $G$ is a locally compact group 
and $\pi$ is a unitary representation of $G$ on a Hilbert
space $\cH_\pi$, then the von Neumann algebra generated by $\pi$
is $\pi(G)''$, the double commutant of $\pi(G)$ inside
 $\cB(\cH_\pi)$. If 
$\pi(G)''$ is a Type I von Neumann algebra, then $\pi$ is called
a \emph{Type I representation}.

\begin{prop}\label{type_I}
Let $\Gamma$ be a wallpaper group. Then
the $3\Gamma$-wavelet representation
$V$ of $\Gamma_3\rtimes_{\vartheta}\Z$ is a Type I representation.
\end{prop}
\begin{proof}
By Proposition \ref{onto_rep}, 
$V\big(\Gamma_3\rtimes_{\vartheta}\Z\big)=\cG(A,\Gamma)$, where
$A=3\cdot{\rm id}$. By Proposition \ref{abelian_comm},
$V\big(\Gamma_3\rtimes_{\vartheta}\Z\big)'$ is abelian and, thus,
a Type I von Neumann algebra. By Theorem 9.1.3 of \cite{KR}, its
commutant,
$V\big(\Gamma_3\rtimes_{\vartheta}\Z\big)''$, is also Type I. That is, $V$ is a Type I 
representation.
\end{proof} 

The $3\Gamma$-wavelet representation
$V$ is, in some sense, a natural representation of $\Gamma_3\rtimes_{\vartheta}\Z$. Another
natural representation is the left regular 
representation. Our final proposition
concerns the relationship between these two
representations.

\begin{prop}
Let $\Gamma$ be a wallpaper group. Then
the $3\Gamma$-wavelet representation
$V$ and the left regular representation of
 $\Gamma_3\rtimes_{\vartheta}\Z$ are weakly equivalent
 but not equivalent.
\end{prop}
\begin{proof}
The proof of Theorem 3.1 in \cite{lpt} adapts
to this situation to show that $V$ is
weakly equivalent to the left regular
representation of 
$\Gamma_3\rtimes_{\vartheta}\Z$. We note that
$\Gamma_3\rtimes_{\vartheta}\Z$ has no abelian subgroup of finite index, as is easily verified, so Kaniuth's stronger version of
Thoma's Theorem given in \cite{Ka} shows that
the left regular representation of
$\Gamma_3\rtimes_{\vartheta}\Z$ is not
Type I. Since $V$ is Type I, they cannot be
equivalent.
\end{proof}

%%%%%%%%%%%%%%%%%%%%%%%%%%%%%%%%%%%%%%%%%%%%%%%%%%%%


\begin{thebibliography}{99}

\bibitem{cmo} B. Currey, A. Mayeli, and V. Oussa, 
Decompositions of generalized wavelet representations, 
\emph{Contemp. Math.}, {\bf 626} (2014), 67-85.

\bibitem{dl} X. Dai and D.R. Larson, Wandering vectors for unitary systems and orthogonal wavelets, \emph{Mem. Amer. Math. Soc.}, {\bf 134} (640), 1998.

\bibitem{dls} X. Dai, D.R. Larson, and D. Speegle, Wavelet sets in 
$\R^n$, \emph{J. Fourier Anal. Appl.}, {\bf 3},
(1997), 451-456.

\bibitem{dut} D.E. Dutkay, Low-Pass Filters and Representations of the Baumslag Solitar Group, \emph{Trans. Amer. Math. Soc.},
{\bf 358} (2006), 5271-5291.

\bibitem{dhjp} D.E. Dutkay, D. Han, P.E.T. Jorgensen, and G. Picioroaga, On common fundamental domains, \emph{Adv. Math.}, {\bf 239} (2013), 109–127.

\bibitem{ds} D.E. Dutkay, S. Silvestrov,  Reducibility of the wavelet representation associated to the Cantor set, 
\emph{Proc. Amer. Math. Soc.}, {\bf 139} (2011), 3657–3664.
	
\bibitem{fol} G. Folland, \emph{A Course in Abstract Harmonic Analysis}, CRC Press, Boca Raton, 1995.

\bibitem{gcm} A.L. Gonz\'alez and M. del Carmen Moure, Crystallographic Haar Wavelets, \emph{J. Fourier Analysis and Applications} {\bf 17} (2011), 1119-1137.

\bibitem{hr} E. Hewitt and K.A. Ross, \emph{Abstract Harmonic Analysis}, Springer-Verlag, New York, 1979.  

\bibitem{hul} A. Hulanicki, Means and Folner conditions on locally
compact groups, \emph{Studia Math.} {\bf 27} (1966), 87-104.

\bibitem{KR} R.V. Kadison and J.R. Ringrose, 
\emph{Fundamentals of the theory of operator algebras, Vol. II, Advanced theory.} Pure and Applied Mathematics, {\bf 100}, Academic Press, Inc., Orlando, FL, 1986.

\bibitem{Ka} E. Kaniuth, Der Typ der regul\"aren Darstellung diskreter Gruppen, \emph{Math. Ann.} {\bf 182} (1969), 334-339.

\bibitem{KT} E. Kaniuth and K.F. Taylor, \emph{Induced Representations of Locally Compact Groups}, Cambridge Tracts in Mathematics v. \textbf{197}, Cambridge University Press, Cambridge, 2012.

\bibitem{lpt} L.-H. Lim, J.A. Packer, and K.F. Taylor, A direct integral decomposition of the wavelet representation, \emph{Proc. Amer. Math. Soc.} \textbf{129} (2001), 3057--3067.

\bibitem{mt} J. MacArthur and K.F. Taylor, Wavelets with crystal symmetry shifts. \emph{J. Fourier Analysis and Applications} 
{\bf 17} (2011), 1109-1118.

\bibitem{mac}G.W. Mackey, Imprimitivity for representations of locally compact groups I,\emph{Proc. Nat. Acad. Sci. U.S.A.} \textbf{35} (1949), 537-545.

\bibitem{mac2}G.W. Mackey, \emph{The Theory of Unitary Group Representations}, Chicago Lectures in Mathematics, University of Chicago Press, Chicago, IL., 1976.

\bibitem{m1} K.D. Merrill, Simple wavelet sets for scalar dilations in $\R^2$, \emph{Representations, wavelets, and frames}, 
\emph{Appl. Numer. Harmon. Anal.}, Birkh\"auser Boston, Boston, MA, (2008), 177-192. 

\bibitem{m2} K.D. Merrill, Simple wavelet sets for matrix dilations in $\R^2$, \emph{Numer. Funct. Anal. Optim.} {\bf 33} (2012), 1112-1125.

\bibitem{m3} K.D. Merrill, Simple wavelet sets in $\R^n$ 
\emph{J. Geom. Anal.}, {\bf 25} (2015), 1295-1305.

\bibitem{KM} K.D. Merrill, Wavelet sets for crystallographic groups, preprint.

\bibitem{mor} P.J. Morandi, \emph{The Classification of Wallpaper Patterns:  From Cohomology to Escher's Tesselations,} New Mexico State University, Las Cruces, 2007.

\bibitem{ni} O.A. Nielsen, \emph{Direct integral theory}, 
Lecture Notes in Pure and Applied Mathematics, {\bf 61} Marcel Dekker, Inc., New York, 1980.

\bibitem{pat} A.L.T. Paterson, \emph{Amenability}, Mathematical Surveys and Monographs, {\bf 29} American Mathematical Society, Providence, RI, 1988.

\bibitem{sch} D. Schattschneider, The plane symmetry groups, their recognition and notation, \emph{ Amer. Math. Monthly} \textbf{85} (1978), 439--450.

\end{thebibliography}
\end{document}